\documentclass{amsart}
\usepackage{amsfonts}

\usepackage{amsmath}
\usepackage{geometry}
\usepackage{graphicx}
\usepackage{float}
\usepackage{subfig}

\numberwithin{equation}{section}
\newtheorem{theorem}{Theorem}[section]

\newtheorem{corollary}{Corollary}[section]

\newtheorem{definition}{Definition}[section]
\newtheorem{example}{Example}[section]

\newtheorem{lemma}{Lemma}[section]

\newtheorem{problem}[theorem]{Problem}
\newtheorem{proposition}{Proposition}[section]
\newtheorem{remark}{Remark}

\begin{document}
\title{The Exchange Value Embedded In A Transport System}
\author{Qinglan Xia}
\address[Q. Xia]{University of California at Davis\\
Department of Mathematics\\
Davis, CA, 95616}
\email{qlxia@math.ucdavis.edu}
\urladdr{http://math.ucdavis.edu/\symbol{126}qlxia}
\thanks{This work is supported by an NSF grant DMS-0710714.}
\author{Shaofeng Xu}
\address[S. Xu]{University of California at Davis\\
Department of Economics\\
Davis, CA, 95616}
\email{sxu@ucdavis.edu}
\maketitle

\begin{abstract}
\noindent This paper shows that a well designed transport system has an
embedded exchange value by serving as a market for potential exchange
between consumers. Under suitable conditions, one can improve the welfare of
consumers in the system simply
by allowing some exchange of goods between consumers during
transportation without incurring additional transportation costs. We propose an
explicit valuation formula to measure this exchange value for a given
compatible transport system. This value is always nonnegative and bounded
from above. Criteria based on transport structures, preferences and prices
are provided to determine the existence of a positive exchange value.
Finally, we study a new optimal transport problem with an objective taking
into account of both transportation cost and exchange value.
\end{abstract}

\section{Introduction}

A transport system is used to move goods from sources to targets. In
building such a system, one typically aims at minimizing the total
transportation cost. This consideration has motivated the theoretical
studies of many optimal transport problems. For instance, the well-known
Monge-Kantorovich problem (e.g. \cite{Ambrosio}, \cite{Brenier}, \cite%
{caffarelli}, \cite{evan2}, \cite{mccann}, \cite{otto}, \cite{mtw}, \cite{monge}, \cite%
{villani}) studies how to find an optimal transport map or transport plan
between two general probability measures with the optimality being measured
by minimizing some cost function. Applications of the Monge-Kantorovich problem
to economics may be found in the literature such as 
\cite{kantorovich}, \cite{buttazzo} and \cite{mccann1}. The present paper gives
another application by introducing the economics notion of an ``exchange
value'' which is suitable for a ramified transport system.
\textit{Ramified optimal transportation}
has been recently proposed and studied (e.g. \cite{gilbert}, \cite{xia1}, \cite{msm}, \cite{xia2}
, \cite{BCM}, \cite{buttazzo}, \cite{xia4}, \cite{Solimini}, \cite{paolini}, %
\cite{book}, \cite{xia5}, \cite{xia6}) to model a branching transport
system. An essential feature of such a transportation is to favor
transportation in groups via a cost function which depends concavely on
quantity. Transport systems with such branching structures are observable
not only in nature as in trees, blood vessels, river channel networks,
lightning, etc. but also in efficiently designed transport systems such as
used in railway configurations and postage delivery networks. Those studies
have focused on the cost value of a branching transport system in terms of
its effectiveness in reducing transportation cost.

In this article, we show that there is another value, named as \textit{\
exchange value}, embedded in some ramified transport systems.
\begin{figure}[h]
\centering
\subfloat[
$G_1$]{\label{G_1}\includegraphics[width=0.4\textwidth,
height=2.25in]{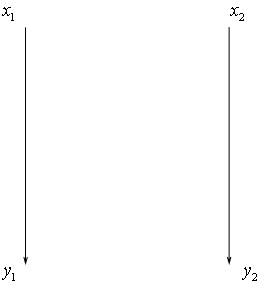}} \hspace{0.5in}
\subfloat[
$G_2$]{\label{G_2}\includegraphics[width=0.4\textwidth, height=2.25in]{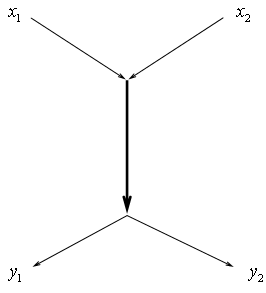}}
\caption{Unlike a traditional transport system $G_{1}$, a ramified transport
system $G_{2}$ provides an exchange value.}
\end{figure}
As an illustration, we consider a spacial economy with two goods located at two
distinct points $\left\{ x_{1},x_{2}\right\} $ and two consumers living at
two different locations $\left\{ y_{1},y_{2}\right\} $. The spacial
distribution is shown in Figure 1. Suppose consumer 1 favors good 2 more
than good 1. However, good 2 may be more expensive than good 1 for some
reason such as a higher transportation fee. As a result, she buys good 1
despite the fact that it is not her favorite. On the contrary, consumer 2
favors good 1 but ends up buying good 2, as good 1 is more expensive than
good 2 for him. Given this purchase plan, a traditional transporter will
ship the ordered items in a transport system like $G_{1}$ (see Figure \ref%
{G_1}). However, as shown in \cite{xia1} etc, a transport system like $G_{2}$
(see Figure \ref{G_2}) with some branching structure might be more cost
efficient than $G_{1}$. One may save some transportation cost by using a
transport system like $G_{2}$ instead of $G_{1}$. Now, we observe another
very interesting phenomenon about $G_{2}$. When using this transport system,
one can simply switch the items which leads to consumer 1 getting good 2 and
consumer 2 receives good 1. This exchange of items makes both consumers
better off since they both get what they prefer. More importantly, no extra
transportation cost is incurred during this exchange process. In other
words, a ramified transport system like $G_{2}$ may possess an exchange
value, which cannot be found in other transport systems like $G_{1}$.

The \textit{exchange value} concept of a transport system that we propose
here is valuable for both economics and mathematics. Existing market
theories (e.g. \cite{Arrow}, \cite{Debreu}, \cite{Debreu2}, \cite{Eaton}, %
\cite{Koopmans}, \cite{Lange}, \cite{Mas-Colell}, \cite{Samuelson}) focus on
the mechanism of exchanges between economic agents on an abstract market
with relatively few discussions on its form. Our study complements the
existing theories by showing that a transport system actually serves as a
concrete market whose friction for exchange depends on the structure of the
transport system as well as factors like preferences, prices, spatial
distribution, etc. The existence of such an exchange value is due to the
fact that the transport system provides a medium for potential exchange
between agents. From the perspective of mathematical theory on optimal
transport problem, our study provides another rationale for ramified
structure which usually implies a potential exchange value. Furthermore, a
new optimality criterion needs to be considered when building a transport
system which leads to a new mathematical problem. Instead of simply
minimizing the transportation cost, one might have to minimize the
difference between transportation cost and exchange value.

The remainder of this paper is organized as follows. Section 2 describes the
model environment with a brief review of consumer's problem and related
concepts from ramified optimal transportation. Sections 3 and 4 contain the
main results of the paper. Section 3 proposes an explicit valuation formula
to measure the exchange value for a given compatible transport system. The
exchange value is defined by solving a maximization problem, which has a
unique solution under suitable conditions. Criteria based on transport
structures, preferences and prices are provided to determine the existence
of a positive exchange value. We show that a reasonable combinations of
these factors guarantees a positive exchange value. Section 4 studies a new
optimal transport problem with an objective taking into account of both
transportation cost and exchange value.

In this paper, we will use the following notations:

\begin{itemize}
\item $X$: a compact convex subset of a Euclidean space $\mathbb{R}^{m}$.

\item $\mathbb{R}_{+}^{k}$: a subset of $\mathbb{R}^{k}$ defined as $\left\{
\left( x_{1},...,x_{k}\right) \in \mathbb{R}^{k}:x_{i}\geq 0,\text{ }
i=1,...,k\right\} .$

\item $\mathbb{R}_{++}^{k}$: a subset of $\mathbb{R}^{k}$ defined as $
\left\{ \left( x_{1},...,x_{k}\right) \in \mathbb{R}^{k}:x_{i}>0,\text{ }
i=1,...,k\right\} .$

\item $p_{j}$: a price vector in $\mathbb{R}_{++}^{k}$ faced by consumer $j$
, $j=1,...,\ell .$

\item $q_{j}$: a consumption vector in $\mathbb{R}_{+}^{k}$ of consumer $j$,
$j=1,...,\ell .$

\item $\mathcal{E}$: an economy as defined in (\ref{economy}).

\item $\bar{q}$: the consumption plan as defined in (\ref{q_bar}).

\item $e_{j}\left( p_{j},\tilde{u}_{j}\right) $: the expenditure function of
consumer $j$, $j=1,...,\ell $, as defined in (\ref{ex_min}).

\item $\mathbf{a}$: the atomic measure representing sources of goods, see ( %
\ref{source}).

\item $\mathbf{b}$: the atomic measure representing consumers, see (\ref%
{consumer}).

\item $G$: a transport path from $\mathbf{a}$ to $\mathbf{b}$.

\item $q$: a transport plan from $\mathbf{a}$ to $\mathbf{b}$.

\item $S(q)$: the total expenditure function as defined in (\ref{S_function}%
).

\item $\Omega \left( \bar{q}\right) $: the set of all transport paths
compatible with $\bar{q}$, as defined in (\ref{Omega}).

\item $\mathcal{F}_{G}$: the set of all feasible transport plans of $G$ as
defined in (\ref{feasible}).

\item $\mathcal{V}\left( G\right) $: the exchange value of a transport path $%
G$, as defined in (\ref{main_problem}).

\item $\mathbf{M}_{\alpha }\left( G\right) $: the transportation cost of a
transport path $G$ as defined in (\ref{M_a_cost}).
\end{itemize}

\section{Consumer's Problem and Ramified Optimal Transportation}

\subsection{Consumer's Problem}

Suppose there are $k$ sources of different goods which could be purchased by
$\ell $ consumers distributed on $X$. Each source $x_{i}\in X$ supplies only
one type of goods, $i=1,...,k$. Each consumer $j$ located at $y_{j}\in X$
derives utility from consuming $k$ goods according to a utility function $%
u_{j}:\mathbb{R}_{+}^{k}\rightarrow \mathbb{R}:\left(
q_{1j},...,q_{kj}\right) \mapsto u_{j},$ $j=1,...,\ell ,$ where $u_{j}:%
\mathbb{R}_{+}^{k}\rightarrow \mathbb{R}$ is \textit{continuous, concave and
increasing}$,$ $j=1,...,\ell .$ Each consumer $j$ has an initial wealth $%
w_{j}>0$ and faces a price vector $p_{j}=\left( p_{1j},...,p_{kj}\right) \in
\mathbb{R}_{++}^{k},$ $j=1,...,k.$ We allow the prices to vary across
consumers to accommodate the situation where consumers on different
locations may have to pay different prices for the same good. This variation
could be possibly due to different transportation fees. We denote this
economy as
\begin{equation}
\mathcal{E}=\left( U,P,W;x,y\right) .  \label{economy}
\end{equation}

Now, we give a brief review of a consumer's decision problem. Discussions of
these materials can be found in most advanced microeconomics texts (e.g. %
\cite{Mas-Colell}). Each consumer $j$ will choose an utility maximizing
consumption plan given the price $p_{j}$ and wealth $w_{j}.$ More precisely,
the consumption plan $\bar{q}_{j}$ is derived from the following utility
maximizing problem:
\begin{equation}
\bar{q}_{j}\in \arg \max \left\{ u_{j}\left( q_{j}\right) \text{ }|\text{ }%
q_{j}\in \mathbb{R}_{+}^{k},\text{ }p_{j}\cdot q_{j}\leq w_{j}\right\} .
\label{q_bar}
\end{equation}%
Given the continuity and concavity of $u_{j},$ we know this problem has a
solution.

As will be used in defining the exchange value, we also consider the
expenditure minimizing problem for a given utility level $\tilde{u}%
_{j}>u_{j}\left( \mathbf{0}\right) $:
\begin{equation}
e_{j}\left( p_{j},\tilde{u}_{j}\right) =\min \left\{ p_{j}\cdot q_{j}\text{ }%
|\text{ }q_{j}\in \mathbb{R}_{+}^{k},\text{ }u_{j}\left( q_{j}\right) \geq
\tilde{u}_{j}\right\} ,  \label{ex_min}
\end{equation}%
which is actually a problem dual to the above utility maximization problem.
The continuity and concavity of $u_{j}$ guarantee a solution to this
minimization problem. Here, $e_{j}\left( p_{j},\tilde{u}_{j}\right) $
represents the minimal expenditure needed for consumer $j$ to reach a
utility level $\tilde{u}_{j}.$ Since $\tilde{u}_{j}>u_{j}\left( \mathbf{0}%
\right) $, we know that $e_{j}\left( p_{j},\tilde{u}_{j}\right) >0.$ Lemma %
\ref{Properties_e} (see \cite{Mas-Colell}) shows several standard properties
of the expenditure function $e_{j}$.

\begin{lemma}
\label{Properties_e}Suppose that $u_{j}$ is a continuous, increasing utility
function on $\mathbb{R}_{+}^{k}.$ The expenditure function $e_{j}\left(
p_{j},\tilde{u}_{j}\right) $ is

\begin{enumerate}
\item Homogeneous of degree one in $p_{j}.$

\item Strictly increasing in $\tilde{u}_{j}$ and nondecreasing in $p_{ij}$
for any $i=1,...,k.$

\item Concave in $p_{j}.$

\item Continuous in $p_{j}$ and $\tilde{u}_{j}.$
\end{enumerate}
\end{lemma}

The following lemma shows a nice property of $e_{j}$ when $u_{j}$ is
homogeneous. This property will be used in the next section to characterize
the solution set of the maximization problem defining exchange value.

\begin{lemma}
\label{homogeneous}If $u_{j}:\mathbb{R}_{+}^{k}\rightarrow \mathbb{R}$ is
homogeneous of degree $\beta _{j}>0$, then $e_{j}\left( p_{j},\tilde{u}%
_{j}\right) $ is homogeneous of degree $\frac{1}{\beta _{j}}$ in $\tilde{u}%
_{j}$, which implies
\begin{equation*}
e_{j}\left( p_{j},\tilde{u}_{j}\right) =e_{j}\left( p_{j},1\right) \left(
\tilde{u}_{j}\right) ^{\frac{1}{\beta _{j}}}.
\end{equation*}
\end{lemma}

\begin{proof}
For any $\lambda >0$, since $u_{j}$ is homogeneous of degree $\beta _{j}$,
we have
\begin{eqnarray*}
e_{j}\left( p_{j},\lambda \tilde{u}_{j}\right) &=&\min \left\{ p_{j}\cdot
q_{j}\text{ }|\text{ }q_{j}\in \mathbb{R}_{+}^{k},\text{ }u_{j}\left(
q_{j}\right) \geq \lambda \tilde{u}_{j}\right\} \\
&=&\min \left\{ p_{j}\cdot q_{j}\text{ }|\text{ }q_{j}\in \mathbb{R}_{+}^{k},%
\text{ }u_{j}\left( \left( 1/\lambda \right) ^{1/\beta _{j}}q_{j}\right)
\geq \tilde{u}_{j}\right\} \\
&=&\min \left\{ \left( \lambda \right) ^{1/\beta _{j}}p_{j}\cdot \tilde{q}%
_{j}\text{ }|\text{ }\tilde{q}_{j}\in \mathbb{R}_{+}^{k}\text{, }u_{j}\left(
\tilde{q}_{j}\right) \geq \tilde{u}_{j}\right\} \text{, where }\tilde{q}%
_{j}=(1/\lambda) ^{1/\beta _{j}}q_{j}, \\
&=&\left( \lambda \right) ^{1/\beta _{j}}e_{j}\left( p_{j},\tilde{u}%
_{j}\right) .
\end{eqnarray*}%
Therefore, $e_{j}\left( p_{j},\tilde{u}_{j}\right) $ is homogeneous of
degree $\frac{1}{\beta _{j}}$ in $\tilde{u}_{j}.$
\end{proof}

\subsection{Ramified Optimal Transportation}

Let $X$ be a compact convex subset of a Euclidean space $\mathbb{R}^{m}$.
Recall that a Radon measure $\mathbf{a}$ on $X$ is \textit{atomic} if $%
\mathbf{a}$ is a finite sum of Dirac measures with positive multiplicities.
That is
\begin{equation*}
\mathbf{a}=\sum\limits_{i=1}^{k}m_{i}\delta _{x_{i}}
\end{equation*}%
for some integer $k\geq 1$ and some points $x_{i}\in X$, $m_{i}>0$ for each $%
i=1,\cdots ,k$.\

In the environment of the previous section, the $k$ sources of goods can be
represented as an atomic measure on $X$ by
\begin{equation}
\mathbf{a}=\sum\limits_{i=1}^{k}m_{i}\delta _{x_{i}},\text{ where }%
m_{i}=\sum_{j=1}^{\ell }\bar{q}_{ij}\text{,}  \label{source}
\end{equation}%
where $\bar{q}_{j}=\left( \bar{q}_{1j},\cdots ,q_{kj}\right) $ is given by (%
\ref{q_bar}). Also, the $\ell $ consumers can be represented by another
atomic measure on $X$ by
\begin{equation}
\mathbf{b}=\sum\limits_{j=1}^{\ell }n_{j}\delta _{y_{j}}\text{, where }%
n_{j}=\sum_{i=1}^{k}\bar{q}_{ij}\text{.}  \label{consumer}
\end{equation}%
Without loss of generality, we may assume that
\begin{equation*}
\sum_{ij}\bar{q}_{ij}=1,
\end{equation*}%
and thus both $\mathbf{a}$ and $\mathbf{b}$ are probability measures on $X$.

\begin{definition}
(\cite{xia1}) A \textit{transport path from }$\mathbf{a}$\textit{\ to }$%
\mathbf{b}$ is a weighted directed graph $G$ consists of a vertex set $%
V\left( G\right) $, a directed edge set $E\left( G\right) $ and a weight
function $w:E\left( G\right) \rightarrow \left( 0,+\infty \right) $ such
that $\{x_{1},x_{2},...,x_{k}\}\cup \{y_{1},y_{2},...,y_{\ell }\}\subseteq
V(G)$ and for any vertex $v\in V(G)$, there is a balance equation
\begin{equation}
\sum_{e\in E(G),e^{-}=v}w(e)=\sum_{e\in E(G),e^{+}=v}w(e)+\left\{
\begin{array}{c}
m_{i},\text{\ if }v=x_{i}\text{\ for some }i=1,...,k \\
-n_{j},\text{\ if }v=y_{j}\text{\ for some }j=1,...,\ell \\
0,\text{\ otherwise }%
\end{array}%
\right.  \label{balance}
\end{equation}%
where each edge $e\in E\left( G\right) $ is a line segment from the starting
endpoint $e^{-}$ to the ending endpoint $e^{+}$.
\end{definition}

Note that the balance equation (\ref{balance}) simply means the conservation
of mass at each vertex. Viewing $G$ as a one dimensional polyhedral chain,
we have the equation $\partial G=b-a$.

Let
\begin{equation*}
Path\left( \mathbf{a,b}\right)
\end{equation*}%
be the space of all transport paths from $\mathbf{a}$ to $\mathbf{b}$.

\begin{definition}
(e.g. \cite{Ambrosio}, \cite{villani}) A \textit{transport plan} from $%
\mathbf{a}$ to $\mathbf{b}$ is an atomic probability measure
\begin{equation}
q=\sum_{i=1}^{k}\sum_{j=1}^{\ell }q_{ij}\delta _{\left( x_{i},y_{j}\right) }
\label{transport_plan}
\end{equation}%
in the product space $X\times X$ such that
\begin{equation}
\sum_{i=1}^{k}q_{ij}=n_{j}\text{ and }\sum_{j=1}^{\ell }q_{ij}=m_{i}
\label{margins}
\end{equation}%
for each $i$ and $j$. Denote $Plan\left( \mathbf{a},\mathbf{b}\right) $ as
the space of all transport plans from $\mathbf{a}$ to $\mathbf{b}$.
\end{definition}

For instance, the $\bar{q}$ given by (\ref{q_bar}) is a transport plan in $%
Plan\left( \mathbf{a,b}\right) $.

Now, we want to consider the compatibility between a pair of transport path
and transport plan (\cite{xia1}, \cite{book}). Let $G$ be a given transport
path in $Path\left( \mathbf{a,b}\right) $. From now on, we assume that for
each $x_{i}$ and $y_{j}$, there exists at most one directed polyhedral curve
$g_{ij}$ from $x_{i}$ to $y_{j}$. In other words, there exists a list of
distinct vertices
\begin{equation}
V\left( g_{ij}\right) :=\left\{ v_{i_{1}},v_{i_{2}},\cdots ,v_{i_{h}}\right\}
\label{V_g}
\end{equation}%
in $V\left( G\right) $ with $x_{i}=v_{i_{1}}$, $y_{j}=v_{i_{h}}$, and each $%
\left[ v_{i_{t}},v_{i_{t+1}}\right] $ is a directed edge in $E\left(
G\right) $ for each $t=1,2,\cdots ,h-1$. For some pairs of $\left(
i,j\right) $, such a curve $g_{ij}$ from $x_{i}$ to $y_{j}$ may fail to
exist, due to reasons like geographical barriers, law restrictions, etc. If
such curve does not exists, we set $g_{ij}=0$ to denote the {\it empty} directed polyhedral curve. By doing so, we construct a
matrix
\begin{equation}
g=\left( g_{ij}\right) _{k\times \ell }  \label{g_matrix}
\end{equation}%
with each element of $g$ being a polyhedral curve. A very simple example
satisfying these conditions is illustrated in Figure \ref{simple_graph}.
\begin{figure}[h]
\centering \includegraphics[width=0.6\textwidth, height=1.5in]{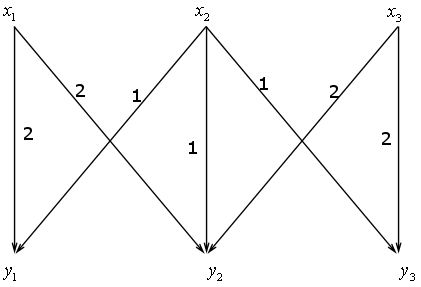}
\caption{A transport path from $4\delta_{x_1}+3\delta_{x_2}+4\delta_{x_3}$ to
$3\delta_{y_1}+5\delta_{y_2}+3\delta_{y_3}$ with $g_{13}=0, g_{31}=0$.}
\label{simple_graph}
\end{figure}

\begin{definition}
A pair $\left( G,q\right) $ of a transport path $G\in Path\left( \mathbf{a,b}%
\right) $ and a transport plan $q\in Plan\left( \mathbf{a,b}\right) $ is
compatible if $q_{ij}=0$ whenever $g_{ij}$ does not exist and
\begin{equation}
G=q\cdot g.  \label{compatible_pair}
\end{equation}
\end{definition}

Here, the equation (\ref{compatible_pair}) means
\begin{equation*}
G=\sum_{i=1}^{k}\sum_{j=1}^{\ell }q_{ij}g_{ij}.
\end{equation*}

In terms of edges, it says that for each edge $e\in E\left( G\right) $, we
have
\begin{equation*}
\sum_{e\subseteq g_{ij}}q_{ij}=w\left( e\right) .
\end{equation*}

\begin{example}
\label{G_bar} Let $x^{\ast }\in X\setminus \left\{ x_{1},\cdots
,x_{k},y_{1},\cdots ,y_{\ell }\right\} $. We may construct a path $\bar{G}%
\in Path\left( \mathbf{a,b}\right) $ as follows. Let
\begin{eqnarray*}
V\left( \bar{G}\right) &=&\left\{ x_{1},\cdots ,x_{k}\right\} \cup \left\{
y_{1},\cdots ,y_{\ell }\right\} \cup \left\{ x^{\ast }\right\} , \\
E\left( \bar{G}\right) &=&\left\{ \left[ x_{i},x^{\ast }\right] :i=1,\cdots
,k\right\} \cup \left\{ \left[ x^{\ast },y_{j}\right] :j=1,\cdots ,\ell
\right\} ,
\end{eqnarray*}%
and
\[
w\left( \left[ x_{i},x^{\ast }\right] \right) =m_{i}, 
w\left( \left[ x^{\ast },y_{j}\right] \right) =n_{j}
\]
for each $i$ and $j$. In this case, each $g_{ij}$ is the union of two edges $%
\left[ x_{i},x^{\ast }\right] \cup \left[ x^{\ast },y_{j}\right] $. Then,
each transport plan $q\in Plan\left( \mathbf{a,b}\right) $ is compatible
with $\bar{G}$ because
\begin{equation*}
\sum_{\left[ x_{i^{\ast }},x^{\ast }\right] \subseteq
g_{ij}}q_{ij}=\sum_{j=1}^{\ell }q_{i^{\ast }j}=m_{i^{\ast }}=w\left( \left[
x_{i^{\ast }},x^{\ast }\right] \right)
\end{equation*}%
and
\begin{equation*}
\sum_{\left[ x^{\ast },y_{j^{\ast }}\right] \subseteq
g_{ij}}q_{ij}=\sum_{i=1}^{k}q_{ij^{\ast }}=n_{j^{\ast }}=w\left( \left[
x^{\ast },y_{j^{\ast }}\right] \right) .
\end{equation*}
\end{example}

\section{Exchange Value Of A Transport system}

In a transport system, a transporter can simply ship the desired bundle to
consumers as they have initially planned. This is a universal strategy.
However, we will see that allowing the exchange of goods between consumers
may make them better off without incurring any additional transportation
cost. In other words, there is an \textit{exchange value} embedded in some
transport system.

\subsection{Exchange Value}

For each probability measure $q=\left( q_{ij}\right) \in \mathcal{P}\left(
X\times X\right) $, we define

\begin{equation}
S\left( q\right) =\sum_{j=1}^{\ell }e_{j}\left( p_{j},u_{j}\left(
q_{j}\right) \right) =\sum_{j=1}^{\ell }\min \left\{ p_{j}\cdot t_{j}\text{ }%
|\text{ }t_{j}\in \mathbb{R}_{+}^{k},\text{ }u_{j}\left( t_{j}\right) \geq
u_{j}\left( q_{j}\right) \right\} ,  \label{S_function}
\end{equation}%
where $q_{j}=\left( q_{1j},q_{2j},...,q_{kj}\right) $ for each $j=1,\cdots
,\ell $. Here, $S\left( q\right) $ represents the least total expenditure
for each individual $j$ to reach utility level $u_{j}\left( q_{j}\right) .$
One can simply use Lemmas \ref{Properties_e} and \ref{homogeneous} to prove
the following lemma which shows several properties of this function $S.$

\begin{lemma}
\label{Properties_S}Suppose each $u_{j}$ is continuous, concave, and
increasing on $\mathbb{R}_{+}^{k},$ $j=1,...,\ell .$ The function $S\left(
q\right) $ is

\begin{enumerate}
\item Homogeneous of degree one in $p=\left( p_{1},...,p_{\ell }\right) .$

\item Increasing in $q$ and nondecreasing in $p_{ij}$ for any $i=1,...,k,$ $%
j=1,...,\ell .$

\item Concave in $p.$

\item Continuous in $p$ and $q.$
\end{enumerate}
\end{lemma}

Let $\bar{q}\in Plan\left( \mathbf{a,b}\right) $ be the initial plan given
by (\ref{q_bar}). Denote
\begin{equation}
\Omega \left( \bar{q}\right) =\left\{ G\in Path\left( \mathbf{a,b}\right)
\text{ }|\text{ }\left( G,\bar{q}\right) \text{ is compatible}\right\} .
\label{Omega}
\end{equation}

Let $G\in \Omega \left( \bar{q}\right) $ be fixed and $g=\left(
g_{ij}\right) $ be the corresponding matrix of $G$ as given in (\ref%
{g_matrix}). That is,
\begin{equation*}
G=g\cdot \bar{q}.
\end{equation*}%
Then, we introduce the following definition:

\begin{definition}
Each transport plan in the set%
\begin{equation}
\mathcal{F}_{G}=\left\{ q\in \mathcal{P}\left( X\times X\right) \left|
\begin{array}{c}
q\text{ is compatible with }G \\
u_{j}\left( q_{j}\right) \geq u_{j}\left( \bar{q}_{j}\right) ,\text{ }%
j=1,...,\ell .%
\end{array}%
\right. \right\}  \label{feasible}
\end{equation}%
is called a feasible plan for $G.$
\end{definition}

Recall that $q$ is compatible with $G$ means that
\begin{equation}
q_{ij}=0\text{ if }g_{ij}\text{ does not exist}  \label{zero_g_ij}
\end{equation}%
and
\begin{equation*}
g\cdot q=g\cdot \bar{q},
\end{equation*}%
in the sense that for each edge $e\in E\left( G\right) $, we have an
equality
\begin{equation}
\sum_{e\subseteq g_{ij}}q_{ij}=w\left( e\right) \text{, where }w\left(
e\right) =\sum_{e\subseteq g_{ij}}\bar{q}_{ij}.  \label{edge_equation}
\end{equation}%
For any feasible plan $q\in \mathcal{F}_{G}$, the constraint $u_{j}\left(
q_{j}\right) \geq u_{j}\left( \bar{q}_{j}\right) $ means that $q_{j}$ is at
least as good as $\bar{q}_{j}$ for each consumer $j$.

Since $\bar{q}\in Plan\left( \mathbf{a,b}\right) $, the compatibility
condition automatically implies that $q\in Plan\left( \mathbf{a,b}\right) $
whenever $q\in \mathcal{F}_{G}$.

\begin{lemma}
$\mathcal{F}_{G}$ is a nonempty, convex and compact subset of $\mathcal{P}%
\left( X\times X\right) .$
\end{lemma}

\begin{proof}
Clearly, $\bar{q}\in \mathcal{F}_{G},$ showing that $\mathcal{F}_{G}\neq
\emptyset .$ The set $\mathcal{F}_{G}$ is convex since it is an intersection
of two convex sets $\left\{ q\in \mathcal{P}\left( X\times X\right) \text{ }|%
\text{ }g\cdot q=G\text{ }\right\} $ and $\prod\nolimits_{j=1}^{\ell
}\left\{ q_{j}\in \mathcal{P}\left( X\times X\right) \left| u_{j}\left(
q_{j}\right) \geq u_{j}\left( \bar{q}_{j}\right) \right. \right\} ,$ where
the convexity of $\left\{ q_{j}\in \mathcal{P}\left( X\times X\right) \left|
u_{j}\left( q_{j}\right) \geq u_{j}\left( \bar{q}_{j}\right) \right.
\right\} $ comes from the concavity of $u_{j}$, $j=1,...,\ell .$ Since each $%
u_{j}$ is continuous, we have $\mathcal{F}_{G}$ is a closed subset of $%
\mathcal{P}\left( X\times X\right) $ and hence it is compact.
\end{proof}

Note that when $G=\bar{G}$ as constructed in the example \ref{G_bar}, we
have
\begin{equation*}
\mathcal{F}_{\bar{G}}=\left\{ q\in Plan\left( \mathbf{a,b}\right) \left|
u_{j}\left( q_{j}\right) \geq u_{j}\left( \bar{q}_{j}\right) ,\text{ }%
j=1,...,\ell \right. \right\} .
\end{equation*}%
Clearly, for each $G\in Path\left( \mathbf{a,b}\right) $, we have
\begin{equation}
\bar{q}\in \mathcal{F}_{G}\subseteq \mathcal{F}_{\bar{G}}.
\label{comparison_G_bar}
\end{equation}

\begin{definition}
Let $\mathcal{E}$ be an economy as in (\ref{economy}). For each transport
path $G\in \Omega \left( \bar{q}\right) $, we define the exchange value of $%
G $ by
\begin{equation}
\mathcal{V}\left( G;\mathcal{E}\right) =\underset{q\in \mathcal{F}_{G}}{\max
}\text{ }S\left( q\right) -S\left( \bar{q}\right) ,  \label{main_problem}
\end{equation}%
where $S$ is given by (\ref{S_function}). Without causing confusion, we may
simply denote $\mathcal{V}\left( G;\mathcal{E}\right) $ by $\mathcal{V}%
\left( G\right) $.
\end{definition}

Since $S$ is a continuous function on a compact set, the exchange value
function $\mathcal{V}:\Omega \rightarrow \lbrack 0,\infty )$ is well
defined. Furthermore, for each $q\in \mathcal{F}_{G}$, given $u_{j}\left(
q_{j}\right) \geq u_{j}\left( \bar{q}_{j}\right) $ for all $j$, we have
\begin{equation}
S\left( q\right) \geq S\left( \bar{q}\right) .  \label{S_q_comparison}
\end{equation}

\begin{remark}
Our way of defining the feasibility set $\mathcal{F}_{G}$ guarantees that
the exchange value is not obtained at the cost of increasing transportation
cost. This is because the compatibility condition ensures that replacing $%
\bar{q}$ by any feasible plan $q\in \mathcal{F}_{G}$ will not change the
transportation cost $M_{\alpha }\left( G\right) $ (to be defined later in (%
\ref{M_a_cost})), as the quantity on each edge $e$ of $G$ is set to be $%
w\left( e\right) $.
\end{remark}

The following proposition shows that the exchange value is always
nonnegative and bounded from above.

\begin{proposition}
\label{upperbound}For any $G\in \Omega \left( \bar{q}\right) ,$%
\begin{equation*}
0\leq \mathcal{V}\left( G\right) \leq \mathcal{V}\left( \bar{G}\right) .
\end{equation*}
\end{proposition}

\begin{proof}
This follows from the definition as well as (\ref{comparison_G_bar}).
\end{proof}

\begin{example}
Let's return to the example discussed in introduction. More precisely,
suppose $u_{1}\left( q_{11},q_{21}\right) =q_{11}+3q_{21},$ $w_{1}=1/2,$ $%
p_{1}=\left( 1,6\right) $ and $u_{2}\left( q_{12},q_{22}\right)
=3q_{12}+q_{22},$ $w_{2}=1/2,$ $p_{2}=\left( 6,1\right) .$ By solving (\ref%
{q_bar}), i.e.
\begin{eqnarray*}
\bar{q}_{1} &\in &\arg \max \left\{ u_{1}\left( q_{11},q_{21}\right) \text{ }%
|\text{ } p_{1}\cdot q_{1}\leq w_{1}\right\} \\
&=&\arg \max \left\{ q_{11}+3q_{21} \text{ }|\text{ } q_{11}+6q_{21}\leq
1/2\right\} \\
&=&\left\{ \left( 1/2,0\right) \right\} ,
\end{eqnarray*}%
we find \ $\bar{q}_{1}=\left( 1/2,0\right) $. Similarly, we have $\bar{q}%
_{2}=\left( 0,1/2\right) $. This gives the initial plan
\begin{equation*}
\bar{q}=\left(
\begin{array}{cc}
1/2 & 0 \\
0 & 1/2%
\end{array}%
\right) .
\end{equation*}%
Now, solving expenditure minimization problems (\ref{ex_min}) yields
\begin{eqnarray*}
e_{1}\left( p_{1},\tilde{u}_{1}\right) &=&\min \left\{ p_{1}\cdot q_{1}\text{
}|\text{ }q_{1}\in \mathbb{R}_{+}^{2},u_{1}\left( q_{1}\right) \geq \tilde{u}%
_{1}\right\} \\
&=&\min \left\{ q_{11}+6q_{21}\text{ }|\text{ }\left( q_{11},q_{21}\right)
\in \mathbb{R}_{+}^{2},q_{11}+3q_{21}\geq \tilde{u}_{1}\right\} \\
&=&\tilde{u}_{1}.
\end{eqnarray*}%
Similarly, we have $e_{2}\left( p_{2},\tilde{u_{2}}\right) =\tilde{u_{2}}$.
From these, we get
\begin{equation}
S\left( q\right) =e_{1}\left( p_{1},u_{1}\left( q_{1}\right) \right)
+e_{2}\left( p_{2},u_{2}\left( q_{2}\right) \right) =u_{1}\left(
q_{1}\right) +u_{2}\left( q_{2}\right)  \notag
\end{equation}%
for each probability measure $q\in \mathcal{P}\left( X\times X\right) $.
Now, we find the exchange value embedded in the transport systems $G_{1}$
and $G_{2}$ as given in Figure 1.

\begin{itemize}
\item $G_{1}:$ The associated feasible set is%
\begin{equation*}
\mathcal{F}_{G_{1}}=\left\{ q=\left(
\begin{array}{cc}
q_{11} & q_{12} \\
q_{21} & q_{22}%
\end{array}%
\right) \in \mathcal{P}\left( X\times X\right) \left|
\begin{array}{c}
q_{11}=1/2,\text{ }q_{21}=0,\text{ }q_{12}=0,\text{ }q_{22}=1/2, \\
q_{11}+3q_{21}\geq u_{1}\left( \bar{q}_{1}\right) =1/2, \\
\text{ }3q_{12}+q_{22}\geq u_{2}\left( \bar{q}_{2}\right) =1/2.%
\end{array}%
\right. \right\} =\left\{ \bar{q}\right\} .
\end{equation*}%
Thus, the exchange value of $G_{1}$ is
\begin{equation*}
\mathcal{V}\left( G_{1}\right) =\underset{q\in \mathcal{F}_{G_{1}}}{\max }%
S\left( q\right) -S\left( \bar{q}\right) =S\left( \bar{q}\right) -S\left(
\bar{q}\right) =0.
\end{equation*}

\item $G_{2}:$ The associated feasible set is%
\begin{eqnarray*}
\mathcal{F}_{G_{2}} &=&\left\{ q=\left(
\begin{array}{cc}
q_{11} & q_{12} \\
q_{21} & q_{22}%
\end{array}%
\right) \in \mathcal{P}\left( X\times X\right) \left|
\begin{array}{c}
q_{11}+q_{12}=1/2,q_{21}+q_{22}=1/2, \\
q_{11}+q_{21}=1/2, \\
q_{11}+3q_{21}\geq u_{1}\left( \bar{q}_{1}\right) =1/2, \\
\text{ }3q_{12}+q_{22}\geq u_{2}\left( \bar{q}_{2}\right) =1/2.%
\end{array}%
\right. \right\} \\
&=&\left\{ q=\left(
\begin{array}{cc}
q_{11} & 1/2-q_{11} \\
1/2-q_{11} & q_{11}%
\end{array}%
\right) \left|
\begin{array}{c}
q_{11}\leq 1/2 \\
\text{ }q_{11}\geq 0%
\end{array}%
\right. .\right\}
\end{eqnarray*}%
Thus, we have the following exchange value
\begin{eqnarray*}
\mathcal{V}\left( G_{2}\right) &=&\underset{q\in \mathcal{F}_{G_{2}}}{\max }%
S\left( q\right) -S\left( \bar{q}\right) \\
&=&\underset{q\in \mathcal{F}_{G_{2}}}{\max }\left\{ \left(
q_{11}+3q_{21}\right) +\left( 3q_{12}+q_{22}\right) \right\} -1 \\
&=&\underset{0\leq q_{11}\leq \frac{1}{2}}{\max }\left\{ \left(
q_{11}+3\left( 1/2-q_{11}\right) \right) +\left( 3\left( 1/2-q_{11}\right)
+q_{11}\right) \right\} -1 \\
&=&\underset{0\leq q_{11}\leq \frac{1}{2}}{\max }\left\{ 3-4q_{11}\right\}
-1=2.
\end{eqnarray*}
\end{itemize}
\end{example}

Basically, there are three factors affecting the size of exchange value:
transport structures, preferences and prices. In the rest of this section,
we will study how these three factors affect the exchange value.

\subsection{ Transport Structures and Exchange Value}

For any $G\in \Omega \left( \bar{q}\right) $, define
\begin{equation*}
K\left( \bar{q},G\right) =\left\{ q\in \mathcal{P}\left( X\times X\right)
\left| q\text{ is compatible with }G\right. \right\} ,
\end{equation*}%
and
\begin{equation*}
U\left( \bar{q}\right) =\left\{ q\in \mathcal{P}\left( X\times X\right)
\left| u_{j}\left( q_{j}\right) \geq u_{j}\left( \bar{q}_{j}\right) ,\text{ }%
j=1,...,\ell .\right. \right\}
\end{equation*}%
Then,
\begin{equation*}
\mathcal{F}_{G}=K\left( \bar{q},G\right) \cap U\left( \bar{q}\right) .
\end{equation*}%
Clearly, the structure of a transport system influences the exchange value
through $K\left( \bar{q},G\right) .$ For this consideration, this subsection
will focus on the properties of $K\left( \bar{q},G\right) $ whose
implications on exchange value will be self-evident in the following
subsections.

\begin{proposition}
$K\left( \bar{q},G\right) $ is a polygon of dimension $N\left( G\right)
+\chi \left( G\right) -\left( k+\ell \right) $, where $\chi \left( G\right) $
is the Euler Characteristic number of $G$, and $N\left( G\right) $ is the
total number of existing $g_{ij}$'s in $G$.
\end{proposition}

\begin{proof}
For each interior vertex $v$ of $G$, let $\left\{ e_{1},e_{2},\cdots
,e_{h}\right\} \subseteq E\left( G\right) $ be the set of edges with $%
e_{i}^{-}=v$. Then, each $e_{i}$ corresponds to an equation of the form (\ref%
{edge_equation}). Nevertheless, due to the balance equation (\ref{balance}),
we may remove one redundant equation from these $h$ equations. As a result,
the total number of equations of the form (\ref{edge_equation}) equals the
total number of edges of $G$ minus the total number of interior vertices of $%
G$. Thus, $K\left( \bar{q},G\right) $ is defined by $k+\ell -\chi \left(
G\right) $ number of linear equations in the form of (\ref{edge_equation}),
and $\left( k\ell -N\left( G\right) \right) $ number of equations (\ref%
{zero_g_ij}). This shows that $K\left( \bar{q},G\right) $ is a convex
polygon of dimension
\begin{equation}
\dim \left( K\left( \bar{q},G\right) \right) \geq k\ell -\left( k+\ell -\chi
\left( G\right) \right) -\left( k\ell -N\left( G\right) \right) =N\left(
G\right) +\chi \left( G\right) -\left( k+\ell \right) .  \label{dim1}
\end{equation}%
By the following Lemma \ref{dimension_less}, we have an inequality of the
other direction.
\end{proof}

\begin{lemma}
\label{dimension_less}The dimension of $K\left( \bar{q},G\right) $ is no
more than $N\left( G\right) +\chi \left( G\right) -\left( k+\ell \right) $.
\end{lemma}

\begin{proof}
Since $K\left( \bar{q},G\right) $ is defined by $k\ell $ variables $\left(
q_{ij}\right) _{k\times \ell }$ which satisfy equations (\ref{zero_g_ij})
and (\ref{edge_equation}). As the number of equations (\ref{zero_g_ij}) is $%
k\ell -N\left( G\right) $, it is sufficient to show that
\begin{equation*}
rank\left( A\right) \geq k\ell -\left( k\ell -N\left( G\right) \right)
-\left( N\left( G\right) +\chi \left( G\right) -\left( k+\ell \right)
\right) =\left( k+\ell -\chi \left( G\right) \right) ,
\end{equation*}%
where $A$ is the coefficient matrix given by linear equations (\ref%
{edge_equation}). We prove this by using induction on the number $k$. When $%
k=1$, then the coefficient matrix $A^{\left( 1\right) }$ is in the form of%
\begin{equation*}
A^{\left( 1\right) }=\left(
\begin{array}{c}
I \\
B%
\end{array}%
\right) ,
\end{equation*}%
where $I$ is the $N\left( G^{\left( 1\right) }\right) \times N\left(
G^{\left( 1\right) }\right) $ identity matrix $I_{N\left( G^{\left( 1\right)
}\right) }$, and $B$ is some matrix of $N\left( G^{\left( 1\right) }\right) $
columns. Thus, the rank of $A^{\left( 1\right) }$ is $N\left( G^{\left(
1\right) }\right) $. On the other hand, the Euler Characteristic number of $%
G^{\left( 1\right) }$ is
\begin{equation*}
\chi \left( G^{\left( 1\right) }\right) =1+\left( \ell -N\left( G^{\left(
1\right) }\right) \right) \text{,}
\end{equation*}%
which gives
\begin{equation}
rank\left( A^{\left( 1\right) }\right) =\left( 1+\ell \right) -\chi \left(
G^{\left( 1\right) }\right) .  \label{rank_1}
\end{equation}%
Now, we may use induction by assuming that
\begin{equation}
rank\left( A^{\left( k\right) }\right) \geq \left( k+\ell \right) -\chi
\left( G^{\left( k\right) }\right)  \label{rank_k}
\end{equation}%
for any $G^{\left( k\right) }$ from $k$ sources to $\ell $ consumers. We
want to show that
\begin{equation*}
rank\left( A^{\left( k+1\right) }\right) \geq \left( k+\ell +1\right) -\chi
\left( G^{\left( k+1\right) }\right)
\end{equation*}%
for any $G^{\left( k+1\right) }$ from $\left( k+1\right) $ sources to $\ell $
consumers.

Let
\begin{eqnarray*}
E_{1}^{\left( k+1\right) } &=&\left\{ e\in E\left( G^{\left( k+1\right)
}\right) :e\subseteq g_{ij}\text{ for some }i\in \left\{ 1,\cdots ,k\right\}
\text{ and }j\in \left\{ 1,\cdots ,\ell \right\} \right\} \\
E_{2}^{\left( k+1\right) } &=&E\left( G^{\left( k+1\right) }\right)
\setminus E_{1}^{\left( k+1\right) }.
\end{eqnarray*}%
For each $e\in E_{2}^{\left( k+1\right) }$, we know $e\subseteq g_{\left(
k+1\right) j}$ for some $j\in \left\{ 1,\cdots ,\ell \right\} $, but $%
e\notin E_{1}^{\left( k+1\right) }$. Then, for each $e\in E_{1}^{\left(
k+1\right) }$, we have
\begin{equation*}
\sum_{\substack{ 1\leq i\leq k,1\leq j\leq \ell  \\ e\subseteq g_{ij}}}%
q_{ij}+\sum_{\substack{ 1\leq j\leq \ell  \\ e\subseteq g_{\left( k+1\right)
j}}}q_{\left( k+1\right) j}=w\left( e\right) .
\end{equation*}%
Also, for each $e\in E_{2}^{\left( k+1\right) }$, we have
\begin{equation*}
\sum_{\substack{ 1\leq j\leq \ell  \\ e\subseteq g_{\left( k+1\right) j}}}%
q_{\left( k+1\right) j}=w\left( e\right) .
\end{equation*}%
As a result, the matrix $A^{\left( k+1\right) }$ can be expressed in the
form
\begin{equation}
A^{\left( k+1\right) }=\left(
\begin{array}{cc}
A^{\left( k\right) } & B^{\left( k+1\right) } \\
0 & C^{\left( k+1\right) }%
\end{array}%
\right) .  \label{A_(k+1)}
\end{equation}%
Now, we consider a new transport path
\begin{equation*}
\tilde{G}=\sum_{e\in E_{2}^{\left( k+1\right) }}w\left( e\right) \left[ e%
\right] \text{.}
\end{equation*}%
from a single source (i.e. $x_{k+1}$) to a few (say $\tilde{\ell}$ ) targets
(, which do not necessarily belong to the original consumers). The matrix $%
C^{\left( k+1\right) }$ here is the associated $A^{\left( 1\right) }$ matrix
for $\tilde{G}$, and thus has rank $\left( 1+\tilde{\ell}\right) -\chi
\left( \tilde{G}\right) =\tilde{\ell}$ as $\tilde{G}$ is contractible. Also,
we have
\begin{equation*}
\chi \left( G^{\left( k+1\right) }\right) =\chi \left( G^{\left( k\right)
}\right) +1-\tilde{\ell}.
\end{equation*}%
Therefore, by (\ref{rank_k}) and (\ref{A_(k+1)}),
\begin{eqnarray*}
rank\left( A^{\left( k+1\right) }\right) &\geq &rank\left( A^{\left(
k\right) }\right) +rank\left( C^{\left( k+1\right) }\right) \\
&\geq &\left( k+\ell \right) -\chi \left( G^{\left( k\right) }\right)
+\left( 1+\chi \left( G^{\left( k\right) }\right) -\chi \left( G^{\left(
k+1\right) }\right) \right) \\
&=&\left( k+1+\ell \right) -\chi \left( G^{\left( k+1\right) }\right) .
\end{eqnarray*}
\end{proof}

\begin{corollary}
\label{corollary_kl}Suppose $G\in \Omega \left( \bar{q}\right) $.
\end{corollary}

\begin{enumerate}
\item If $k+\ell \geq N\left( G\right) +\chi \left( G\right) $, then $%
\mathcal{F}_{G}=\left\{ \bar{q}\right\} $.

\item If $k+\ell <N\left( G\right) +\chi \left( G\right) $ and $\bar{q}$ is
an interior point of the polygon $K\left( \bar{q},G\right) $, then $\mathcal{%
F}_{G}$ is a convex set of positive dimension. In particular, $\mathcal{F}%
_{G}\neq \left\{ \bar{q}\right\} .$
\end{enumerate}

\begin{proof}
If $k+\ell \geq N\left( G\right) +\chi \left( G\right) $, the convex polygon
$K\left( \bar{q},G\right) $ becomes a dimension zero set, and thus $\mathcal{%
F}_{G}=\left\{ \bar{q}\right\} $. When $k+\ell <N\left( G\right) +\chi
\left( G\right) $, the polygon $K\left( \bar{q},G\right) $ has positive
dimension. Since each $u_{j}$ is concave, $U\left( \bar{q}\right) $ is a
convex set containing $\bar{q}$. When $\bar{q}$ is an interior point of $%
K\left( \bar{q},G\right) $, the intersection $\mathcal{F}_{G}=K\left( \bar{q}%
,G\right) \cap U\left( \bar{q}\right) $ is still a convex set of positive
dimension. Thus, $\mathcal{F}_{G}\neq \left\{ \bar{q}\right\} $.
\end{proof}

\begin{proposition}
\label{pairwise_dis}Suppose $G\in \Omega \left( \bar{q}\right) $ satisfies
the following condition: for any two pairs $\left( i_{1},i_{2}\right) $ with
$i_{1}\neq i_{2}$ and $\left( j_{1},j_{2}\right) $ with $j_{1}\neq j_{2}$,
we have
\begin{equation}
V\left( g_{i_{1}j_{2}}\right) \cap V\left( g_{i_{2}j_{1}}\right) =\emptyset
\text{,}  \label{pairwise_disjoint}
\end{equation}%
where $V\left( g_{ij}\right) $ is given in (\ref{V_g}). Then, $k+\ell \geq
N\left( G\right) +\chi \left( G\right) $. Hence, by Corollary \ref%
{corollary_kl}, $\mathcal{F}_{G}$ is a singleton $\left\{ \bar{q}\right\} $.
\end{proposition}

\begin{proof}
We still use the notations that have been used in the proof of Lemma \ref%
{dimension_less}. When $k=1$, $\chi \left( G\right) =1+\ell -N\left(
G\right) $, and thus $k+\ell =N\left( G\right) +\chi \left( G\right) $. By
using induction, we assume that the result is true for any $k$ sources. We
want to show that it holds for $k+1$ sources. Suppose there are totally $d$
edges of $G^{\left( k+1\right) }$ connecting the vertex $x_{k+1}$, then as
discussed earlier, we may construct a transport path $\tilde{G}$ \ from a
single source $x_{k+1}$ to a few targets $\left\{ v_{1},v_{2},\cdots ,v_{%
\tilde{\ell}}\right\} $ with $v_{i}\in V\left( G^{\left( k\right) }\right) $%
. For each $v_{j}$, it corresponds to a unique $g_{\left( k+1\right) j^{\ast
}}$ for some $j^{\ast }\in \left\{ 1,2,\cdots ,\ell \right\} $ that passing
through the vertex $v_{j}$. Indeed, suppose both $g_{\left( k+1\right)
j_{1}} $ and $g_{\left( k+1\right) j_{2}}$ passing through $v_{j}$ with $%
j_{1}\neq j_{2}$. Since $v_{j}\in V\left( G^{\left( k\right) }\right) $,
there exists an $i^{\ast }\in \left\{ 1,2,\cdots ,k\right\} $ such that $%
x_{i^{\ast }}$ and $v_{j}$ are connected by a directed curve lying in $%
G^{\left( k\right) }$. Then, $v_{j}\in g_{i^{\ast }j_{2}}\cap g_{\left(
k+1\right) j_{1}}$, which contradicts condition (\ref{pairwise_disjoint}).
As a result,
\begin{equation*}
N\left( G^{\left( k+1\right) }\right) =N\left( G^{\left( k\right) }\right) +%
\tilde{\ell}.
\end{equation*}%
On the other hand, it is easy to see that $\chi \left( G^{\left( k+1\right)
}\right) =\chi \left( G^{\left( k\right) }\right) +1-\tilde{\ell}$. So, by
induction,
\begin{eqnarray*}
\left( k+1\right) +\ell &\geq &1+N\left( G^{\left( k\right) }\right) +\chi
\left( G^{\left( k\right) }\right) \\
&=&1+\left( N\left( G^{\left( k+1\right) }\right) -\tilde{\ell}\right)
+\left( \chi \left( G^{\left( k+1\right) }\right) +\tilde{\ell}-1\right)
=N\left( G^{\left( k+1\right) }\right) +\chi \left( G^{\left( k+1\right)
}\right) .
\end{eqnarray*}%
This shows that
\begin{equation*}
k+\ell \geq N\left( G\right) +\chi \left( G\right)
\end{equation*}%
for any $G$ satisfying condition (\ref{pairwise_disjoint}). Therefore, $%
\mathcal{F}_{G}$ is a singleton $\left\{ \bar{q}\right\} $.
\end{proof}

In Proposition \ref{colliner_general}, we will consider an inverse problem
of Proposition \ref{pairwise_dis} under some suitable conditions on the
prices.

Given two transport paths
\begin{eqnarray*}
G_{1} &=&\left\{ V\left( G_{1}\right) ,E\left( G_{1}\right) ,w_{1}:E\left(
G_{1}\right) \rightarrow \lbrack 0,+\infty )\right\} \text{ and} \\
G_{2} &=&\left\{ V\left( G_{2}\right) ,E\left( G_{2}\right) ,w_{2}:E\left(
G_{2}\right) \rightarrow \lbrack 0,+\infty )\right\} ,
\end{eqnarray*}%
we say $G_{1}$ is topologically equivalent to $G_{2}$ if there exists a
homeomorphism $h:X\rightarrow X$ such that
\begin{eqnarray*}
V\left( G_{2}\right) &=&h\left( V\left( G_{1}\right) \right) , \\
E\left( G_{2}\right) &=&\left\{ h\left( e\right) :e\in E\left( G_{1}\right)
\right\} \\
\text{ and }w_{2}\left( h\left( e\right) \right) &=&w_{1}\left( e\right)
\text{ for each }e\in E\left( G_{1}\right) \text{.}
\end{eqnarray*}%
Clearly, if $G_{1}$ is topologically equivalent to $G_{2},$ then $K\left(
\bar{q},G_{1}\right) =K\left( \bar{q},G_{2}\right) $. As a result, we know $%
\mathcal{V}$ is topologically invariant:

\begin{proposition}
\label{topology}If $G_{1}$ is topologically equivalent to $G_{2}$, then $%
\mathcal{V}\left( G_{1}\right) =\mathcal{V}\left( G_{2}\right) $.
\end{proposition}

As will be clear in the next section, the topological invariance of $%
\mathcal{V}$ is a very useful result because it enables us to inherit many
existing theories in ramified optimal transportation when studying a new
optimal transport problem there.

\subsection{Preferences and Exchange Value}

In this subsection, we will study the implications of preferences, which are
represented by utility functions, on the exchange value. The following
proposition shows that there is no exchange value when all consumers derive
their utilities solely from the total amount of goods they consume.

\begin{proposition}
\label{quantity_u}If $u_{j}:\mathbb{R}_{+}^{k}\rightarrow \mathbb{R} $ is of
the form $u_{j}\left( q_{j}\right) =f_{j}\left( \sum_{i=1}^{k}q_{ij}\right) $
for some $f_{j}:[0,\infty )\rightarrow \mathbb{R}$ for each $j=1,...,\ell ,$
then $\mathcal{V}\left( G\right) =0$ for any $G\in \Omega \left( \bar{q}%
\right) $.
\end{proposition}

\begin{proof}
For any $q\in \mathcal{F}_{G},$ by compatibility, we know
\begin{equation*}
\sum_{i=1}^{k}q_{ij}=\sum_{i=1}^{k}\bar{q}_{ij},\text{ }j=1,...,\ell ,
\end{equation*}%
which implies
\begin{equation*}
u_{j}\left( q_{j}\right) =f_{j}\left( \sum_{i=1}^{k}q_{ij}\right)
=f_{j}\left( \sum_{i=1}^{k}\bar{q}_{ij}\right) =u_{j}\left( \bar{q}%
_{j}\right) ,
\end{equation*}%
showing that all consumers find any feasible plan indifferent to $\bar{q}.$
Therefore, we get
\begin{equation*}
\mathcal{V}\left( G\right) =\underset{q\in \mathcal{F}_{G}}{\max }\text{ }%
S\left( q\right) -S\left( \bar{q}\right) =0.
\end{equation*}
\end{proof}

For any $G\in \Omega \left( \bar{q}\right) $, denote $Q\left( G\right) $ as
the solution set of the maximization problem (\ref{main_problem}) defining
exchange value, i.e.,
\begin{equation}
Q\left( G\right) =\left\{ \hat{q}\in \mathcal{F}_{G}\text{ }|\text{ }%
\mathcal{V}\left( G\right) =S\left( \hat{q}\right) -S\left( \bar{q}\right)
\right\} .  \label{Q_G}
\end{equation}%
We are interested in describing geometric properties of the set $Q\left(
G\right) $. In particular, if $Q\left( G\right) $ contains only one element,
then the problem (\ref{main_problem}) has a unique solution.

\begin{proposition}
\label{uniqueness}For any $G\in \Omega \left( \bar{q}\right) ,$

\begin{enumerate}
\item The solution set $Q\left( G\right) $ is a compact nonempty set.

\item If $u_{j}:\mathbb{R}_{+}^{k}\rightarrow \mathbb{R}$ is homogeneous of
degree $\beta _{j}>0$ and $\left( u_{j}\left( q_{j}\right) \right) ^{\frac{1%
}{\beta _{j}}}$ is concave in $q_{j},$ $j=1,...,\ell ,$ then $Q\left(
G\right) $ is convex.

\item (Uniqueness) If $u_{j}:\mathbb{R}_{+}^{k}\rightarrow \mathbb{R}$ is
homogeneous of degree $\beta _{j}>0$ and$\ \left( u_{j}\left( q_{j}\right)
\right) ^{\frac{1}{\beta _{j}}}$ is concave in $q_{j}$ satisfying the
condition
\begin{equation}
\left( u_{j}\left( \left( 1-\lambda _{j}\right) \tilde{q}_{j}+\lambda _{j}%
\hat{q}_{j}\right) \right) ^{\frac{1}{\beta _{j}}}>\left( 1-\lambda
_{j}\right) \left( u_{j}\left( \tilde{q}_{j}\right) \right) ^{\frac{1}{\beta
_{j}}}+\lambda _{j}\left( u_{j}\left( \hat{q}_{j}\right) \right) ^{\frac{1}{%
\beta _{j}}}  \label{nearly_concave}
\end{equation}%
for each $\lambda _{j}\in \left( 0,1\right) $, and any non-collinear $\tilde{%
q}_{j},\hat{q}_{j}\in \mathbb{R}_{+}^{k}$ , for each $j=1,...,\ell .$ Then $%
Q\left( G\right) $ is a singleton, and thus the problem (\ref{main_problem})
has a unique solution.
\end{enumerate}
\end{proposition}

\begin{proof}
Since the function $S\left( q\right) $ is continuous in $q$, $Q\left(
G\right) $ becomes a closed subset of the compact set $\mathcal{F}_{G}$, and
thus $Q\left( G\right) $ is also compact.

If $u_{j}:\mathbb{R}_{+}^{k}\rightarrow \mathbb{R}$ is homogeneous of degree
$\beta _{j}>0$, then Lemma \ref{homogeneous} implies that
\begin{equation}
e_{j}\left( p_{j},u_{j}\left( q_{j}\right) \right) =e_{j}\left(
p_{j},1\right) \left( u_{j}\left( q_{j}\right) \right) ^{\frac{1}{\beta _{j}}%
}\text{ and }S\left( q\right) =\sum\nolimits_{j=1}^{\ell }e_{j}\left(
p_{j},1\right) \left( u_{j}\left( q_{j}\right) \right) ^{\frac{1}{\beta _{j}}%
}.  \label{S_concave}
\end{equation}%
Thus, when each $\left( u_{j}\left( q_{j}\right) \right) ^{\frac{1}{\beta
_{j}}}$ is concave in $q_{j}$, we have $S$ is concave in $q$. Now, for any $%
q^{\ast },\tilde{q}\in Q\left( G\right) $ and $\lambda \in \left[ 0,1\right]
,$ the convexity of $\mathcal{F}_{G}$ implies $\left( 1-\lambda \right)
q^{\ast }+\lambda \tilde{q}\in $ $\mathcal{F}_{G}$ and the concavity of $S$
implies
\begin{equation}
S\left( \left( 1-\lambda \right) q^{\ast }+\lambda \tilde{q}\right) -S\left(
\bar{q}\right) \geq \left( 1-\lambda \right) \left( S\left( q^{\ast }\right)
-S\left( \bar{q}\right) \right) +\lambda \left( S\left( \tilde{q}\right)
-S\left( \bar{q}\right) \right) =\mathcal{V}\left( G\right) ,
\label{S_inequality}
\end{equation}%
showing that $\left( 1-\lambda \right) q^{\ast }+\lambda \tilde{q}\in
Q\left( G\right) .$ Therefore, $Q\left( G\right) $ is convex.

To prove the uniqueness, we note that $\left( 1-\lambda \right) q^{\ast
}+\lambda \tilde{q}\in Q\left( G\right) $ implies an equality in (\ref%
{S_inequality}), i.e.,%
\begin{equation}
\left( u_{j}\left( \left( 1-\lambda \right) q_{j}^{\ast }+\lambda \tilde{q}%
_{j}\right) \right) ^{\frac{1}{\beta _{j}}}=\left( 1-\lambda \right) \left(
u_{j}\left( q_{j}^{\ast }\right) \right) ^{\frac{1}{\beta _{j}}}+\lambda
\left( u_{j}\left( \tilde{q}_{j}\right) \right) ^{\frac{1}{\beta _{j}}}
\label{equality_u_j}
\end{equation}%
for each $\lambda \in \left( 0,1\right) $, and each $j=1,2,\cdots ,\ell $.
When $\left( u_{j}\left( q_{j}\right) \right) ^{\frac{1}{\beta _{j}}}$ is
concave in $q_{j}$ and satisfies (\ref{nearly_concave}), the equality (\ref%
{equality_u_j}) implies that $q_{j}^{\ast }$ and $\tilde{q}_{j}$ are
collinear in the sense that $q_{j}^{\ast }=t_{j}\tilde{q}_{j}$ for some $%
t_{j}\geq 0$. By (\ref{margins}),
\begin{equation*}
n_{j}=\sum_{i}q_{ij}^{\ast }=\sum_{i}t_{j}\tilde{q}_{ij}=t_{j}\sum_{j}\tilde{%
q}_{ij}=t_{j}n_{j}.
\end{equation*}%
Therefore, $t_{j}=1$ as $n_{j}>0$. This shows $q^{\ast }=\tilde{q}$ and thus
$Q\left( G\right) $ is a singleton with an element $\tilde{q}.$
\end{proof}

Two classes of utility functions widely used in economics satisfy conditions
in Proposition (\ref{uniqueness}). One is Cobb-Douglas function (\cite%
{Mas-Colell})
\begin{equation*}
u:\mathbb{R}_{+}^{k}\rightarrow \mathbb{R}:u\left( q_{1},...,q_{k}\right)
=\prod\limits_{i=1}^{k}\left( q_{i}\right) ^{\tau _{i}},\text{ }\tau _{i}>0,%
\text{ }i=1,...,k.
\end{equation*}%
The other is Constant Elasticity of Substitution function (\cite{Mas-Colell}%
)
\begin{equation*}
u:\mathbb{R}_{+}^{k}\rightarrow \mathbb{R}:u\left( q_{1},...,q_{k}\right) =%
\left[ \sum_{i=1}^{k}\gamma _{i}\left( q_{i}\right) ^{\tau }\right] ^{\frac{%
\beta }{\tau }},\text{ }\tau \in (0,1),\text{ }\beta >0,\text{ }\gamma
_{i}>0,\text{ }i=1,...,k.
\end{equation*}

\begin{proposition}
\label{V_G_positive}Suppose $u_{j}:\mathbb{R}_{+}^{k}\rightarrow \mathbb{R}$
is homogeneous of degree $\beta _{j}>0$ and $\left( u_{j}\left( q_{j}\right)
\right) ^{\frac{1}{\beta _{j}}}$ is concave in $q_{j}$ satisfying (\ref%
{nearly_concave}) for each $j=1,...,\ell $. For any $G\in \Omega \left( \bar{%
q}\right) ,$ $\mathcal{V}\left( G\right) >0$ if and only if $\mathcal{F}%
_{G}\neq \left\{ \bar{q}\right\} $.
\end{proposition}

\begin{proof}
Clearly, if\ $\mathcal{F}_{G}=\left\{ \bar{q}\right\} $, then $\mathcal{V}%
\left( G\right) =0$. On the other hand, suppose $\mathcal{V}\left( G\right)
=\max_{q\in \mathcal{F}_{G}}S\left( q\right) -S\left( \bar{q}\right) =0$,
then by (\ref{S_q_comparison}), we have
\begin{equation*}
S\left( q\right) =S\left( \bar{q}\right) \text{ for each }q\in \mathcal{F}%
_{G}.
\end{equation*}%
This implies $Q\left( G\right) =\mathcal{F}_{G}$. By proposition \ref%
{uniqueness}, $\mathcal{F}_{G}$ is a singleton $\left\{ \bar{q}\right\} $.
\end{proof}

This proposition says that each transport path $G\in \Omega \left( \bar{q}%
\right) $ has a positive exchange value as long as $\mathcal{F}_{G}$
contains more than one element.

\begin{theorem}
\label{theorem 1}Suppose $u_{j}:\mathbb{R}_{+}^{k}\rightarrow \mathbb{R}$ is
homogeneous of degree $\beta _{j}>0$ and $\left( u_{j}\left( q_{j}\right)
\right) ^{\frac{1}{\beta _{j}}}$ is concave in $q_{j}$ satisfying (\ref%
{nearly_concave}) for each $j=1,...,\ell $. \ If $k+\ell <N\left( G\right)
+\chi \left( G\right) $ and $\bar{q}$ is an interior point of the polygon $%
K\left( \bar{q},G\right) $, then $\mathcal{V}\left( G\right) >0$.
\end{theorem}

\begin{proof}
This follows from Proposition \ref{V_G_positive} and Corollary \ref%
{corollary_kl}.
\end{proof}

\subsection{Prices and Exchange Value}

In this subsection, we show that one can observe the collinearity in prices
to determine the existence of a positive exchange value.

\begin{proposition}
\label{collinear1}If the price vectors are collinear, i.e., $p_{j}=\lambda
_{j}p_{1},$ for some $\lambda _{j}>0,$ $j=1,...,\ell ,$ then $\mathcal{V}%
\left( G\right) =0$ for any $G\in \Omega \left( \bar{q}\right) $.
\end{proposition}

\begin{proof}
Assume that $\mathcal{V}\left( G\right) >0$. Then we know there exists a
feasible plan $q\in \mathcal{F}_{G}$ such that $u_{j}\left( q_{j}\right)
\geq u_{j}\left( \bar{q}_{j}\right) $, $j=1,...,\ell ,$ with at least one
strict inequality. Without loss of generality, we assume $u_{j^{\ast
}}\left( q_{j^{\ast }}\right) >u_{j^{\ast }}\left( \bar{q}_{j^{\ast
}}\right) .$ For any $j=1,...,\ell ,$ $u_{j}\left( q_{j}\right) \geq
u_{j}\left( \bar{q}_{j}\right) $ implies $p_{j}\cdot q_{j}\geq p_{j}\cdot
\bar{q}_{j}.$ If not, i.e., $p_{j}\cdot q_{j}<p_{j}\cdot \bar{q}_{j},$ then
by the monotonicity of $u_{j}$, we can find a $\tilde{q}_{j}\in \mathbb{R}%
_{+}^{k}$ such that $\tilde{q}_{j}>q_{j}$,
\begin{equation*}
u_{j}\left( \tilde{q}_{j}\right) >u_{j}\left( q_{j}\right) \geq u_{j}\left(
\bar{q}_{j}\right) \text{ and }p_{j}\cdot \tilde{q}_{j}<p_{j}\cdot \bar{q}%
_{j},
\end{equation*}%
contradicting the assumption that $\bar{q}_{j}$ solves the utility
maximization problem (\ref{q_bar}) of consumer $j.$ Furthermore, for
consumer $j^{\ast }$, by definition of $\bar{q}_{j^{\ast }}$, the inequality
$u_{2}\left( q_{j^{\ast }}\right) >u_{j^{\ast }}\left( \bar{q}_{j^{\ast
}}\right) $ implies $p_{j^{\ast }}\cdot q_{j^{\ast }}>p_{j^{\ast }}\cdot
\bar{q}_{j^{\ast }}.$ Thus, we know $p_{j}\cdot q_{j}\geq p_{j}\cdot \bar{q}%
_{j}$ for all $j$ with a strict inequality for $j=j^{\ast }.$ Since $%
p_{j}=\lambda _{j}p_{1},$ $j=1,...,\ell $, we know $p_{1}\cdot q_{j}\geq
p_{1}\cdot \bar{q}_{j}$ for all $j$ with a strict inequality for $j=j^{\ast
} $. Summing over $j$ yields
\begin{equation*}
\sum_{j=1}^{\ell }p_{1}\cdot q_{j}>\sum_{j=1}^{\ell }p_{1}\cdot \bar{q}_{j}.
\end{equation*}%
Meanwhile, the feasibility of $q$ implies $\sum_{j=1}^{\ell
}q_{j}=\sum_{j=1}^{\ell }\bar{q}_{j}.$ Multiplying both sides by $p_{1}$
leads to
\begin{equation*}
\sum_{j=1}^{\ell }p_{1}\cdot q_{j}=\sum_{j=1}^{\ell }p_{1}\cdot \bar{q}_{j},
\end{equation*}%
a contradiction.
\end{proof}

\begin{corollary}
If there is only one good ($k=1$) or one consumer ($\ell =1$), then $%
\mathcal{V}\left( G\right) =0$ for any $G\in \Omega \left( \bar{q}\right) $.
\end{corollary}

\begin{proof}
When $k=1,$ define $\lambda _{j}=\frac{p_{j}}{p_{1}}>0,$ $j=1,...,\ell .$
The result follows from Proposition \ref{collinear1}. When $\ell =1,$ for
any $G\in \Omega \left( \bar{q}\right) ,$ the feasible set is
\begin{equation*}
\mathcal{F}_{G}=\left\{ q_{1}=\left( q_{11},\cdots ,q_{k1}\right) \in
Plan\left( \mathbf{a,b}\right) \left| q_{i1}=m_{i}=\bar{q}_{i1}\text{ \ for
each }i\right. \right\} =\left\{ \bar{q}_{1}\right\} ,
\end{equation*}%
which clearly yields $\mathcal{V}\left( G\right) =0.$
\end{proof}

\begin{proposition}
\label{collinear2}Let $k=2$ and $\ell =2$. Suppose $u_{j}$ is differentiable
at $\bar{q}$ with $\nabla u_{j}\left( \bar{q}_{j}\right) >0,$ $j=1,2$ and $%
\bar{q}_{ij}>0$ for each $i,j$. If $G\in \Omega \left( \bar{q}\right) $ with
\begin{equation}
V\left( g_{12}\right) \cap V\left( g_{21}\right) \neq \emptyset ,\text{ and }%
p_{21}>p_{11},p_{12}>p_{22},  \label{price_condition}
\end{equation}%
then $\mathcal{V}\left( G\right) >0$.
\end{proposition}

\begin{proof}
Since $g_{12}$ and $g_{21}$ overlap, we denote $\gamma _{2}$ to be the curve
where $g_{12}$ and $g_{21}$ overlap with endpoints $z_{1}$ and $z_{2}$. Let $%
\gamma _{1}$,$\gamma _{3}$, $\gamma _{4}$ and $\gamma _{5}$ be the
corresponding curves from $x_{1}$ to $z_{1}$, $z_{2}$ to $y_{1}$, $x_{2}$ to
$z_{1},$ and $z_{2}$ to $y_{2}$ respectively. Then, these $\gamma _{i}$'s
are disjoint except at their endpoints. See Figure \ref{fig2}. Now, we may
express $g_{ij}$'s as
\begin{eqnarray*}
g_{11} &=&\gamma _{1}+\gamma _{2}+\gamma _{3}, \\
g_{21} &=&\gamma _{4}+\gamma _{2}+\gamma _{3}, \\
g_{12} &=&\gamma _{1}+\gamma _{2}+\gamma _{5}, \\
g_{22} &=&\gamma _{4}+\gamma _{2}+\gamma _{5},
\end{eqnarray*}%
which imply
\begin{equation}
g_{11}+g_{22}=g_{12}+g_{21}.  \label{equation_g}
\end{equation}%
\begin{figure}[h!]
\centering \includegraphics[width=0.4\textwidth, height=2in]{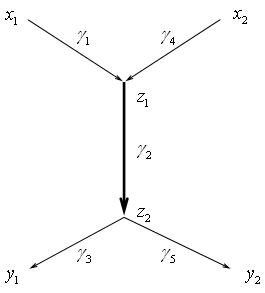}
\caption{A positive exchange value}
\label{fig2}
\end{figure}

Now, let
\begin{equation*}
\tilde{q}=\bar{q}+\left(
\begin{array}{cc}
-\epsilon & \epsilon \\
\epsilon & -\epsilon%
\end{array}%
\right) ,
\end{equation*}%
where $\epsilon $ is a sufficiently small positive number. Then, by (\ref%
{equation_g}),
\begin{eqnarray*}
g\cdot \tilde{q} &=&g\cdot \left( \bar{q}+\left(
\begin{array}{cc}
-\epsilon & \epsilon \\
\epsilon & -\epsilon%
\end{array}%
\right) \right) \\
&=&g\cdot \bar{q}+\epsilon \left( -g_{11}+g_{12}+g_{21}-g_{22}\right)
=g\cdot \bar{q},
\end{eqnarray*}%
which shows that $\tilde{q}$ is compatible with $G$. Now, we show $%
u_{1}\left( \tilde{q}_{1}\right) >u_{1}\left( \bar{q}_{1}\right) $. Since $%
\bar{q}_{1}=\left( \bar{q}_{11},\bar{q}_{21}\right) \in \mathbb{R}_{++}^{2}$
is derived from the utility maximization problem (\ref{q_bar}) of consumer $%
1,$ it must satisfy the first order condition at $\bar{q}_{1}$:%
\begin{equation}
\partial u_{1}\left( \bar{q}_{1}\right) /\partial q_{1}=\lambda p_{11}\text{
and }\partial u_{1}\left( \bar{q}_{1}\right) /\partial q_{2}=\lambda p_{21}
\label{FOC1}
\end{equation}%
for some $\lambda >0$. Thus, using Taylor's Theorem, we have
\begin{eqnarray*}
u_{1}\left( \tilde{q}_{1}\right) &=&u_{1}\left( \bar{q}_{1}\right) +\frac{%
\partial u_{1}\left( \bar{q}_{1}\right) }{\partial q_{11}}\left( \tilde{q}%
_{11}-\bar{q}_{11}\right) +\frac{\partial u_{1}\left( \bar{q}_{1}\right) }{%
\partial q_{21}}\left( \tilde{q}_{21}-\bar{q}_{21}\right) +o\left( \epsilon
\right) \\
&=&u_{1}\left( \bar{q}_{1}\right) +\lambda p_{11}\left( -\epsilon \right)
+\lambda p_{21}\epsilon +o\left( \epsilon \right) \text{, by (\ref{FOC1})} \\
&=&u_{1}\left( \bar{q}_{1}\right) +\lambda \epsilon \left(
p_{21}-p_{11}\right) +o\left( \epsilon \right) \\
&>&u_{1}\left( \bar{q}_{1}\right) \text{, by (\ref{price_condition}).}
\end{eqnarray*}%
Similarly, we have $u_{2}\left( \tilde{q}_{2}\right) >u_{2}\left( \bar{q}%
_{2}\right) $. This shows that $\tilde{q}\in \mathcal{F}_{G}$. By Lemma \ref%
{Properties_e}, we have $S\left( \tilde{q}\right) >S\left( \bar{q}\right) $,
and thus
\begin{equation*}
\mathcal{V}\left( G\right) =\underset{q\in \mathcal{F}_{G}}{\max }\text{ }%
S\left( q\right) -S\left( \bar{q}\right) \geq S\left( \tilde{q}\right)
-S\left( \bar{q}\right) >0.
\end{equation*}
\end{proof}

\begin{theorem}
\label{colliner_general}Suppose $u_{j}$ is differentiable at $\bar{q}$ with $%
\nabla u_{j}\left( \bar{q}_{j}\right) \in \mathbb{R}_{++}^{k},$ $%
j=1,...,\ell ,$ and $\bar{q}\in \mathbb{R}_{++}^{k\ell }.$ If there exists
some $i_{1}\neq i_{2}\in \left\{ 1,...,k\right\} $, $j_{1}\neq j_{2}\in
\left\{ 1,...,\ell \right\} $ satisfies
\begin{equation*}
p_{i_{2}j_{1}}>p_{i_{1}j_{1}},p_{i_{1}j_{2}}>p_{i_{2}j_{2}}\text{ and }%
V\left( g_{i_{1}j_{2}}\right) \cap V\left( g_{i_{2}j_{1}}\right) \neq
\emptyset
\end{equation*}%
for $G\in \Omega \left( \bar{q}\right) $, then $\mathcal{V}\left( G\right)
>0 $.
\end{theorem}

\begin{proof}
This follows from an analogous proof of proposition \ref{collinear2}, as
shown in Figure \ref{fig3}.
\end{proof}

\begin{figure}[h!]
\label{picture3} \centering \includegraphics[width=0.45\textwidth]{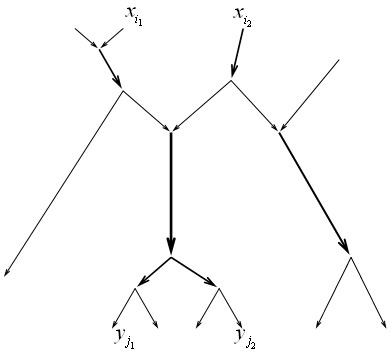}
\caption{A ramified transport system with positive exchange value.}
\label{fig3}
\end{figure}

To conclude this section, we've seen how transport structures, preferences
and prices jointly determine the exchange values. Each of these factors may
lead to a zero exchange value under very rare situations. More precisely,
when the structure of the transport system yields a singleton feasible set $%
\mathcal{F}_{G}$ (Corollary \ref{corollary_kl}, Proposition \ref%
{pairwise_dis})$,$ or the utility functions are merely quantity dependent
(Proposition \ref{quantity_u}), or price vectors are collinear across
consumers (Proposition \ref{collinear1}), the exchange value is zero.
However, under more regular situations, there exists a positive exchange
value for a ramified transport system. For instance, if the utility
functions satisfy the conditions in (3) of Theorem \ref{theorem 1} with a
non-singleton feasible set $\mathcal{F}_{G}$ (Theorem \ref{theorem 1}) or
the transport systems are of ramified structures with some non collinear
price vectors (Theorem \ref{colliner_general}), there exists a positive
exchange value.

\section{A New Optimal Transport Problem}

In the previous section, we have considered the exchange value $\mathcal{V}%
\left( G\right) $ for any $G\in \Omega \left( \bar{q}\right) $. A natural
question would be whether there exists a $G^{\ast }$ that maximizes $%
\mathcal{V}\left( G\right) $ among all $G\in \Omega \left( \bar{q}\right) $.
The answer to this question has already been provided in Proposition \ref%
{upperbound} as the particular transport path $\bar{G}\in \Omega \left( \bar{%
q}\right) $ is an obvious maximizer. However, despite the fact that $\bar{G}$
maximizes exchange value, it may be inefficient when accounting for
transportation cost. Nevertheless, as indicated previously, one should not
neglect the benefit of obtaining an exchange value from a transport system.
As a result, it is reasonable to consider both transportation cost and
exchange value together when designing a transport system.

Recall that in \cite{xia1} etc, a ramified transport system is modeled by a
transport path between two probability measures $\mathbf{a}$ and $\mathbf{b}$%
. For each transport path $G\in Path\left( \mathbf{a,b}\right) $ and any $%
\alpha \in \left[ 0,1\right] $, the $\mathbf{M}_{\alpha }$\textbf{\ }cost of
$G$ is defined by
\begin{equation}
\mathbf{M}_{\alpha }\left( G\right) :=\sum_{e\in E\left( G\right) }w\left(
e\right) ^{\alpha }length\left( e\right) .  \label{M_a_cost}
\end{equation}%
When $\alpha <1$, a \textquotedblleft Y-shaped\textquotedblright\ path from
two sources to one target is usually more preferable than a
\textquotedblleft V-shaped\textquotedblright\ path. In general, a transport
path with a branching structure may be more cost efficient than the one with
a \textquotedblleft linear\textquotedblright\ structure. A transport path $%
G\in Path\left( \mathbf{a,b}\right) $ is called an $\alpha -$\textit{optimal
transport path} if it is an $\mathbf{M}_{\alpha }$ minimizer in $Path\left(
\mathbf{a,b}\right) $.

Based on the above discussions, we propose the following minimization
problem.

\begin{problem}
Given two atomic probability measures $\mathbf{a}$ and $\mathbf{b}$ on $X$
in an economy $\mathcal{E}$ given by (\ref{economy}), find a minimizer of
\begin{equation}
H_{\alpha ,\sigma }\left( G\right) :=M_{\alpha }\left( G\right) -\sigma
\mathcal{V}\left( G\right)  \label{new transport}
\end{equation}%
among all $G\in \Omega \left( \bar{q}\right) $, where $\Omega \left( \bar{q}%
\right) $ is given by (\ref{Omega}), and $\alpha \in \lbrack 0,1)$ and $%
\sigma \geq 0$ are fixed constants.
\end{problem}

When the utility functions are merely quantity dependent (Proposition \ref%
{quantity_u}) or when price vectors are collinear across consumers
(Proposition \ref{collinear1}), the exchange value of any $G\in \Omega
\left( \bar{q}\right) $ is always zero. In these cases, $H_{\alpha ,\sigma
}\left( G\right) =M_{\alpha }\left( G\right) $ for any $\sigma $. Thus, the
study of $H_{\alpha ,\sigma }$ coincides with that of $M_{\alpha }$, which
can be found in existing literature (e.g. \cite{xia1}, \cite{book}).
However, as seen in the previous section, it is quite possible that $%
H_{\alpha ,\sigma }$ does not agree with $M_{\alpha }$ on $\Omega \left(
\bar{q}\right) $ for $\sigma >0$ in a general economy $\mathcal{E}$.

As $\mathcal{V}$ is topologically invariant (Proposition \ref{topology}),
many results that can be found in literature about $M_{\alpha }$ still hold
for $H_{\alpha ,\sigma }$. For instance, the Melzak algorithm for finding an
$M_{\alpha }$ minimizer (\cite{melzak}, \cite{gilbert}, \cite{book}) in a
fixed topological class still applies to $H_{\alpha ,\sigma }$ because $%
\mathcal{V}\left( G\right) $ is simply a constant within each topological
class. Also, as the balance equation (\ref{balance}) still holds, one can
still calculate angles between edges at each vertex using existing formulas (%
\cite{xia1}), and then get a universal upper bound on the degree of vertices
on an optimal $H_{\alpha ,\sigma }$ path.

However, due to the existence of exchange value, one may possibly favor an
optimal $H_{\alpha ,\sigma }$ path instead of the usual optimal $M_{\alpha }$
path when designing a transport system. The topological type of the optimal $%
H_{\alpha ,\sigma }$ path may differ from that of the optimal $M_{\alpha }$
path. This observation is illustrated by the following example.

\begin{figure}[h]
\centering
\subfloat[
$G_1$]{\includegraphics[width=0.3\textwidth,
height=1.6in]{Fig1a.PNG}\label{ga}} \hspace{0.25in}
\subfloat[
$G_2$]{\includegraphics[width=0.3\textwidth, height=1.6in]{Fig1b.PNG}\label{gb}}
\hspace{0.25in}
\subfloat[
$G_3$]{\includegraphics[width=0.3\textwidth, height=1.6in]{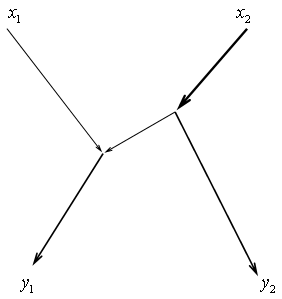}\label{gc}}
\caption{Three topologically different transport systems.}
\label{fig5}
\end{figure}
\begin{example}
Let us consider the transportation from two sources to two consumers.
If we only consider minimizing $\mathbf{M}_{\alpha }$ transportation cost,
each of the three topologically different types shown in Figure \ref{fig5}
may occur. However, when $\sigma $ is sufficiently large, only $G_{2}$ in
Figure \ref{gb} may be selected under suitable conditions of $u$ and $p$.
This is because $G_{2}$ has a positive exchange value which does not exist
in either $G_{1}$ or $G_{3}$.
\end{example}

\end{document}